\RequirePackage[table]{xcolor}                                                                                                                                                                                                                                                                                                                                                                                                                                                                                                                                                                                                                                                                                                                                                                                                                                                                                                                                                                                                                                                                                                                                                                                                                                                                                                                                                                                                                                                                                                                                                                                                                                                                                                                                                                                                                                                                                                                                                                                                                                                                                                                                                                                                                                                                                                                                                                                                                                                                                                                                                                                                                                                                                                                                                                                                                                                                                                                           
\documentclass[a4paper,12pt]{amsart}
\usepackage{CJK}
\usepackage{romannum}
\usepackage{amsfonts}
\usepackage{mathtools}
\usepackage{amsmath,amscd}

\usepackage{ifthen}
\usepackage{amsrefs}
\usepackage{mathrsfs}
\usepackage{amsthm}
\usepackage{amssymb}

\usepackage{tikz-cd}
\usepackage{graphicx}
\usepackage{relsize}

\usepackage{MnSymbol}

\usepackage{hyperref}
\usepackage{url}
\usepackage{etoolbox}
\usepackage{rotating}

\usepackage[shortlabels]{enumitem}
\usepackage[paper=a4paper,left=20mm,right=20mm,top=25mm,bottom=30mm]{geometry}

\usepackage{tikz}
\usetikzlibrary{backgrounds}
\usepackage{tkz-euclide}
\usepackage{xcolor}
\usetikzlibrary{trees,snakes,shapes.geometric}
\usepackage[outline]{contour}
\contourlength{1.5pt}
\usetikzlibrary{positioning}
\usetikzlibrary{%
  matrix,%
  calc,%
  arrows%
}

\setlist[enumerate]{topsep=0em, itemsep= -0em, parsep = 0 em, label=$(\alph*)$}
\let\emptyset\varnothing
\nocite{*}

\newcommand{\cA}{\mathcal{A}}

\newcommand{\cC}{\mathcal{C}}
\newcommand{\cD}{\mathcal{D}}

\newcommand{\cK}{\mathcal{K}}
\newcommand{\cM}{\mathcal{M}}
\newcommand{\cN}{\mathcal{N}}

\newcommand{\cW}{\mathcal{W}}

\newcommand{\cR}{\mathcal{R}}

\DeclareMathOperator{\rad}{rad}
\DeclareMathOperator{\rep}{rep}

\DeclareMathOperator{\coker}{coker}

\DeclareMathOperator{\modd}{mod}
\DeclareMathOperator{\supp}{supp}

\DeclareMathOperator{\im}{im}

\DeclareMathOperator{\Irr}{Irr}

\DeclareMathOperator{\EIP}{EIP}

\DeclareMathOperator{\Inj}{Inj}
\DeclareMathOperator{\Sur}{Surj}

\DeclareMathOperator{\EKP}{EKP}

\let\emptyset\varnothing
\newtheorem{proposition}{Proposition}[section]
\newtheorem{Theorem}[proposition]{Theorem}
\newtheorem{Lemma}[proposition]{Lemma}

\newtheorem{corollary}[proposition]{Corollary}

\newtheorem{example}{Example}

\newenvironment{Remark}[1][Remark.]{\begin{trivlist}
\item[\hskip \labelsep {\bfseries #1}]}{\end{trivlist}}

\newenvironment{Definition}[1][Definition.]{\begin{trivlist}
\item[\hskip \labelsep {\bfseries #1}]}{\end{trivlist}}

\author{Jie Liu}
\address{School of Mathematics and Statistics, Guangdong University of Technology,
Guangzhou 510520, People’s Republic of China}

\email{jie@gdut.edu.cn}

\address{Shenzhen International Center for Mathematics,   Southern University of Science and Technology, shenzhen 518055, People’s Republic of China }

\title{Widths of regular components for n-regular tree $T(n)$} 

\begin{document}

\rmfamily

%%%%%%%%%%%%%%%%%%%%%% TITELSEITE %%%%%%%%%%%%%%%%%%%%%%%%%%%%%%%%

\pagenumbering{arabic}
\thispagestyle{empty}

\begin{abstract}
Let $(T(n),\Omega)$ be the covering of the generalized Kronecker quiver $K(n)$, where $\Omega$ is a bipartite orientation. Given a regular Auslander--Reiten  component   $\cD$  of  $\modd(T(n),\Omega)$,  we introduce two invariants: the width $\cW(\cD)$ and the number of flow modules $b(\cD)$.  We show that   $\cW(\cD)\geq \frac{b(\cD)+1}{2}$. In particular, we get $\{\cW(\cD)| \cD \text{ is a regular component} \}=\mathbb{N}$.

\end{abstract}

\maketitle

\providecommand{\keywords}[1]
{
  \small	
  \textbf{\text{Keywords:}} #1
}


\section{introduction}

\vspace{.3cm}

Let $k$ be an algebraically closed  field.   The generalized  Kronecker quiver $K(n)$ is just the quiver with two vertices $1$, $2$, and $n$ arrows $\gamma_1, \cdots, \gamma_n:$ $1\rightarrow 2$. It is popular to study the representations of $K(n)$ when studying finite groups, moduli spaces, quiver Grassmannians and many other aspects of mathematics (cf.\cite{Julia}, \cite{Luis}, \cite{Claus4}). Let $\rep_k(K(n))$ denote the category of the finite-dimensional representations of $K(n)$  and let $\cK_n=kK(n)$, the \textit{path algebra} of $K(n)$ (cf.\ref{Quiver}).  If we use $\modd\cK_n$ to denote the category of finite-dimensional modules over $\cK_n$,  then there  exists an equivalence between the categories  $\rep_k(K(n))$ and  $\modd \cK_n$.   We will  identify these two categories frequently.

  When $n\leq 2$, it is known how to classify  the indecomposable modules in $\modd \cK_n$ \cite[\Romannum{19}.3]{Assem2}. When $n\geq 3$,  however, it is   hopeless to classify all the indecomposable modules since the representation type is wild in this case (cf.\cite[\Romannum{18}]{Assem2}, \cite[1.3]{Kerner}). Hence it is desirable to find invariants in the category $\modd \cK_n$. J. Worch introduced the invariant \textit{width} $\cW(\cC)$ for a regular component $\cC$ of $\modd \cK_n$ in 2013,  and she showed that $\cW(\cC)\geq 0$ \cite[Section 3]{Julia}. C. M. Ringel found two invariants: \textit{the smallest radius} and \textit{the number of flow modules} for the  covering quiver of $K(n)$ in 2018 (cf. \cite{Claus}, \cite{Claus3}, \cite{Bongartz}). Then it is natural to ask what the width of a regular component $\cD$ of the covering is.

From now on, we let $n\geq 3$ for $K(n)$. Let $M\in \modd \cK_n$. Then we have a linear operator $x^M_{\alpha}=\alpha_1M(\gamma_1)+\alpha_2M(\gamma_2)+\cdots+\alpha_nM(\gamma_n): M_1\rightarrow M_2$ for every  $\alpha=(\alpha_1,\cdots,\alpha_n)\in k^n\setminus\{0\}$. Thus we  can define two full  subcategories of $\modd\cK_n$ with the equal kernels property and with the equal images property (cf.\cite[Definition 2.1]{Julia}): $\EKP_n:=\{M\in \modd\cK_n\mid \forall \alpha\in k^n\setminus\{0\}: x^M_{\alpha} \text{ is injective} \}$ and $\EIP_n:=\{M\in \modd\cK_n\mid \forall \alpha\in k^n\setminus\{0\}: x^M_{\alpha} \text{ is surjective} \}.$
One shows  that  $\EIP_n\cap \EKP_n=(0)$, and   the regular Auslander--Reiten (AR) components of the category $\modd\cK_n$  are  of the form $\mathbb{Z}A_{\infty}$ (cf.\cite{Claus1}). Let $M\in \modd \cK_n$ be an indecomposable module and let $\tau$ be the  Auslander--Reiten translation \cite[\Romannum{4}.2.3]{Assem}. Then a  module $M$ is said to be \textit{regular} if $\tau^m(M)\neq 0$ for any $m\in \mathbb{Z}$. Suppose that $M\in  \EKP_n$ is regular. Then there exists some positive integer $m$ such that $\tau^m M\in \EIP_n$ (cf.\cite[Section 3]{Julia}).  Hence it suggests us to find an invariant connecting $\EIP_n\cap \cC$ and $\EKP_n\cap \cC$ for every regular AR component $\cC$ of  $\modd\cK_n$.  Then we can define a \textit{width}  $\cW(\cC)(\geq 0)$ of $\cC$ \cite[Section 3]{Julia}.  In this note, we  investigate the width  from another aspect using covering theory.

 Let  $(T(n),\Omega)$ be the universal covering of $K(n)$ \cite[Section 7]{Daniel}, i.e., $T(n)$ is  an $n$-regular tree  with a bipartite orientation $\Omega$ (see \ref{Quiver}). Let $\modd(T(n),\Omega)$ be the category of finite-dimensional representations of $(T(n),\Omega)$.  Then  $\modd(T(n),\Omega)$ also admits AR--sequences (cf.\cite[2.2]{Bongartz}).  Similarly, we have widths  for the regular components $\cD$  of $\modd(T(n),\Omega)$. C. M. Ringel defined an invariant $b(\cD)$ (the number of flow modules) for the components $\cD$  in 2018 (cf.\cite{Claus}). We study the components $\cD$ and show that $\cW(\cD)> 0$.

\section{preliminaries}
Throughout, $k$ will  denote an algebraically closed field.

\subsection{Trees}
A \textit{graph}   $G=(G_0,G_1)$ consists of two sets: $G_0$ (the set of \textit{vertices}) and $G_1$ (the set of \textit{edges}). Two vertices $a,b\in G_0$  are called  \textit{neighbours}  if $\{a,b\}$ is an edge in $G_1$.    For a graph $G$, by  an \textit{orientation} we mean a map $\Omega$  $:G_1 \rightarrow G_0\times G_0$ such that $\Omega(\{a, a'\})$ is either ($a, a'$) or ($a', a$). We call $(G, \Omega)$  an \textit{oriented graph}. We  write $a\rightarrow a'$ if $\Omega(\{a, a'\})=(a, a')$ and call $a$ the \textit{start}, $a'$ the \textit{target}, $a\rightarrow a'$ the \textit{arrow}.   Meanwhile we call $a$   a \textit{sink} (or a \textit{source}) if it is not a start (or the target, respectively) of any arrow. If any vertex is a sink or a source, then the orientation  $\Omega$ will be called \textit{bipartite}. A \textit{subgraph}    of a  graph $G=(G_0,G_1)$ is a graph  $G'=(G'_0,G'_1)$ such that $G'_0\subseteq G_0, G'_1\subseteq G_1$. A  \textit{walk} of length $t\geq 0$ with \textit{start} $a_0$ and \textit{target} $a_t$ in a graph $G$ is a finite sequence 
\begin{center}
$l_t=(a_0|\{a_0,a_1\},\{a_1,a_2\},\cdots,\{a_{t-1},a_t\}|a_t)$
\end{center}
such that   $a_{i-1}$ and $a_i$ are neighbours for $1\leq i\leq t$, and $a_{j-1}\neq a_{j+1}$ for $1\leq j < t$. We sometimes use $(a_0,a_1,\cdots, a_t)$ to denote the walk $l_t$ if there only exists one walk connecting $a_0$ and $a_t$. The graph $G$  is called \textit{finite} if $G_0$ and $G_1$ are finite sets, and  $G$ is said to be \textit{connected} if for every two vertices $a,b\in G_0$, there always exists a walk that connects them. A   graph $G$ is said to be a \textit{tree} if it is connected and there only exist  walks of length $0$ that connect a vertex  with itself. Moreover, for any two different vertices $a,b$ in a tree $G$ we can only find one walk that  connects them.    Suppose that $G$ is a finite tree. Then we call the walks of the maximal possible length  the \textit{diameter walks} of $G$, and we use $d(G)$ to denote their length. Let $l_t=(a_0,a_1,\cdots,a_t)$ be a  walk in $G$. If  $t=2r$,  then we  call $a_r$ or $\{a_r\}$ the \textit{center} of $l_t$ and $r$ its \textit{radius}. If $t=2r+1$, then we call $\{a_r, a_{r+1}\}$ the \textit{center} of $l_t$ and  $r$ its \textit{radius}. In fact,  all diameter walks of a finite tree $G$ have the same center \cite[2.1]{Claus}. Hence we take  the center and the radius of diameter walks of a finite tree $G$ as the center and the radius of the $G$ itself and we use $C(G)$ and $r(G)$ to denote them, respectively.

\subsection{Quivers}\label{Quiver}

A \textit{quiver} is a quadruple  $Q=(Q_0, Q_1, s, t)$, where $s,t: Q_1\rightarrow Q_0$. The elements in $Q_0$ are called \textit{vertices} or \textit{points}, and the elements in $Q_1$ are called \textit{arrows}, respectively. For every arrow $\alpha\in  Q_1$, we call the vertex $s(\alpha)$  its \textit{start}, and the vertex $t(\alpha)$ its \textit{target}.  Note that  an  oriented graph $(G,\Omega)$ can be seen as a quiver without loops and multiple arrows.

Let $Q$ be a quiver, and let  $x\in Q_0$. We define two sets
\begin{center}
$x^+:=\{y\in Q_0\mid \exists \alpha \in Q_1: t(\alpha)=y \text{ and } s(\alpha)=x \}$,

$x^-:=\{y\in Q_0\mid \exists \alpha \in Q_1: s(\alpha)=y \text{ and } t(\alpha)=x\}.$

\end{center}
The \textit{neighbourhood} $\mathcal{N}(x)$ of $x$ is $x^+\cup x^-.$ We say that $x$  is a  \textit{leaf},  provided $x$ has at most one neighbour. A sequence $(a\mid \alpha_1, \alpha_2,\cdots,\alpha_r\mid b )$ in  $Q$  is called a \textit{path} with start $a$ and target $b$ if $s(\alpha_1)=a,  t(\alpha_r)=b$, and $t(\alpha_k)=s(\alpha_{k+1})$ for each $1\leq k<r$. Moreover, we  take each vertex $x\in Q_0$ as a \textit{trivial path} of length $0$. The quiver $Q$  is said to be \textit{locally bounded} if for every $x\in Q_0$ the neighbourhood $\cN(x)$ is finite and there exists a number $m_x\in \mathbb{N}$ such that  each path in $Q$ which starts or ends in $x$ is of length $\leq m_x$. In this note, we always assume that it is locally bounded when we take about a quiver $Q$.

Let $\delta_{a,b}$ be the Kronecker delta, and let $l_1=(a_0| \alpha_1,\cdots, \alpha_p|a_p),l_2=(b_0|\beta_1,\cdots, \beta_q |b_q)$ be two paths.  We define the product of the paths $l_1$ and $l_2$ by $l_1l_2=\delta_{a_p,b_0}(a_0| \alpha_1,\cdots,\alpha_p,\beta_1,\cdots,\beta_q|b_q)$.  Then we can get a $k$-algebra  $kQ$ of the quiver $Q$  consisting of $k$-linear span of all paths of length $l \geq 0$ (cf. \cite[\Romannum{2}.1]{Assem}). We call $kQ$  the \textit{path algebra} of $Q$. When $Q_0$ is finite, the $kQ$ is an associative and finite-dimensional   $k$-algebra. A finite-dimensional representation $M=((M_x)_{x\in Q_0}, (M(\alpha))_{\alpha\in Q_1})$ of $Q$ consists of vector spaces $M_{x}$ and $k$-linear maps $M(\alpha): M_{s(\alpha)} \rightarrow M_{t(\alpha)}$ such that $\dim_{k}M:=\sum _{x\in Q_0}\dim_k M_{x}$ is finite. A \textit{morphism} $f: M\rightarrow N$ between two representations is a collection of $k$-linear maps $(f_z)_{z\in Q_0}$, such that for each arrow $\alpha: x\rightarrow y$, we have $ f_y \circ M(\alpha)= N(\alpha)\circ f_x$. We often use $\rep_k(Q)$ to denote the category of finite-dimensional representations of $Q$ over $k$. If we use $\modd kQ$ to denote the category of finite-dimensional left modules of $kQ$, then the categories $\modd kQ$ and $\rep_k(Q)$ are equivalent (cf. \cite[\Romannum{3} 1.6]{Assem}).

We use $(T(n),\Omega)$ to denote the universal covering of $K(n)$ with a  bipartite orientation $\Omega$ \cite[Section 7]{Daniel}. Moreover, we use $(T(n),\Omega)_0$ and $(T(n),\Omega)_1$ to denote the sets of vertices and arrows in the quiver $(T(n),\Omega)$, respectively. Then there exists a push-down functor $\pi_\lambda:\rep_k(T(n),\Omega)\rightarrow\rep_k(K(n))$, and $\pi_\lambda$ is exact (cf. \cite{Daniel}). We follow the same notations  as in \cite[Section 7]{Daniel} and give an example in the following.
  
\begin{example}\label{EX1}

Let $M\in\modd(T(3),\Omega)$ be the following 
\begin{center}

\begin{tikzpicture}
\node (00) at (0,0) {$0$};
\node (10) at (1,0) {$k$};
\node (20) at (2,0) {$k$};
\node (30) at (3,0) {$k$};
\node (40) at (4,0) {$k$};
\node (50) at (5,0) {$k$};
\node (60) at (6,0) {$0$,};
\node (1-1) at (1,-1) {$0$};
\node (2-1) at (2,-1) {$k$};
\node (3-1) at (3,-1) {$0$};
\node (4-1) at (4,-1) {$k$};
\node (5-1) at (5,-1) {$0$};

\path [->] (00) edge (10)
           (20) edge node[above]{$\gamma^1_1$}(10) 
           (1-1) edge (10)
          (20) edge node[above]{$\gamma^1_3$} (30)
            (40) edge node[above]{$\gamma^2_1$} (30)               
             (20) edge node[midway, right]{$\gamma^1_2$} (2-1)            
                (40) edge node[above]{$\gamma^2_3$} (50)
                (40) edge node[midway, right]{$\gamma^2_2$} (4-1)
                (60) edge (50)
                (5-1) edge (50)
              (3-1) edge (30)
                      
                           ;
\end{tikzpicture}
\end{center}
where every $M(\gamma^i_j)$ is the identity map and $\pi(\gamma^i_j)=\gamma_j,i,j\in\{1,2,3\}$. In fact,    $M$ is indecomposable \cite[Proposition 1]{Claus5}.  Moreover, we have $\pi_\lambda(M)=N$, where

 \[
N=\begin{tikzcd}
    k^2 
    \arrow[r,  "\gamma_2"]
    \arrow[r, bend left,        "\gamma_1"]
    \arrow[r, bend right, swap, "\gamma_3"]
    &
     k^5,
\end{tikzcd}
\]

 with
\begin{center}
$N(\gamma_1)=\begin{bmatrix}
  1  & 0 \\
  0 & 0\\
  0 & 1\\
  0 &0 \\
  0 & 0
 \end{bmatrix}, N(\gamma_2)=\begin{bmatrix}
 0 & 0\\
 1 & 0 \\
 0 & 0\\
0 & 1\\
0 & 0
 
\end{bmatrix}, N(\gamma_3)=\begin{bmatrix}
0 & 0\\
0& 0\\
1& 0\\
0& 0 \\
0 & 1
\end{bmatrix}.  $ 
  
\end{center}
\end{example}
 \vspace{0.4cm}
 
  Let $\sigma_x$ be the reflection functor of a sink  $x\in (T(n),\Omega)_0$ (cf.\cite[\Romannum{7}.5]{Assem}).  For any two sinks $x,y\in (T(n),\Omega)_0$, we have $\sigma_x \sigma_y=\sigma_y\sigma_x$.  Then we can define a reflection functor $\sigma:$ rep$_k(T(n),\Omega)\rightarrow\rep_k(T(n),\sigma\Omega)$ at all sinks, where $\sigma \Omega$ is the opposite orientation. The functor $\sigma$ is independent of the order used in the composition and acts on the finite-dimensional representations, so that for any non-trivial representation $M\in \rep_k(T(n),\Omega)$ there exist sinks $x_1,x_2,\cdots,x_M$ such that $\sigma(M)=\sigma_{x_1}\sigma_{x_2}\cdots\sigma_{x_M}(M)$. Hence the functor $\sigma$ is  well-defined. We sometimes call the functor $\sigma$ the \textit{shift functor}.  Similarly, we have a reflection functor $\sigma^-$ at  all sources of $(T(n),\Omega)_0$.

   We call the representations in $\rep_k(T(n),\Omega)$  the \textit{graded Kronecker modules} (or \textit{graded modules}, or simply \textit{modules}), and we often use  $\modd(T(n),\Omega)$ to denote the category of  those modules. 
\begin{Definition}
Let $M\in\modd(T(n),\Omega)$ be an indecomposable module. We say that  $M$ is \textit{regular}, provided $\sigma^t(M)\neq 0$ for any $t\in \mathbb{Z}$.

\end{Definition} 
In fact, an indecomposable module $M\in \modd(T(n),\Omega)$ is regular if and only if $\pi_\lambda{(M)}$ is regular (cf.\cite[\Romannum{8}.2.1, \Romannum{8}.2.14]{Assem},  \cite[Corollary 7.2]{Daniel}). Moreover, it can be checked by looking at the dimension vector of $\pi_\lambda(M)$ (cf.\cite{Bo}, \cite{Zhang}). Hence one can show that $M$ is regular since $\pi_\lambda(M)=N$ is regular in Example \ref{EX1}.

The functors $\sigma,\sigma^-$ send regular indecomposable modules to regular indecomposable modules \cite[Lemma 3.1.2]{Jie}. Furthermore, the functors  $\sigma$  and $ \sigma^-$  induce quasi-inverse    equivalences of the $k$-linear full subcategories of $\modd (T(n),\Omega)$ and   $\modd(T(n),\sigma\Omega)$ consisting of  regular modules \cite[Lemma 3.2.7]{Jie}.   Let $M\in\modd(T(n),\Omega)$ be a regular indecomposable module. Then $\sigma\sigma^-(M)\cong \sigma^-\sigma(M)\cong M$ and  $\sigma^2 M\cong\tau M$, where $\tau$ is the Auslander--Reiten translation \cite[2.6]{Claus}.  Let $D_{T(n)}: \modd(T(n),\Omega)\rightarrow\modd(T(n),\sigma\Omega)$ be the duality  as in \cite[7.2]{Daniel}.  Then  we have $D_{T(n)}\circ \sigma (M)\cong \sigma^- \circ D_{T(n)}(M)$, $D_{T(n)}\circ \sigma^- (M)\cong \sigma\circ D_{T(n)}(M)$ \cite[Lemma 3.2.8]{Jie}.

\subsection{Graded Kronecker modules}

 Let $ \emptyset\neq S\subseteq (T(n),\Omega)_0$ be a set of vertices. We use $T(S)$ to denote  the unique minimal tree in $(T(n),\Omega)$ containing $S$.  Let $x\in (T(n),\Omega)_0$,  and let $M\in\modd(T(n),\Omega)$. Define $\supp(M):=\{y\in  (T(n),\Omega)_0\mid M_y\neq (0)\}$, and call it  the \textit{support} of $M$. Let $T(M)=T(\supp(M))$. Then  the vertex $x$ is said to be a \textit{leaf} of $M$,  provided $x$ is a leaf of $T(M)$.  Suppose that $M$ is indecomposable.  Then $T(M)$ is a finite tree, we can  define $d(M)=d(T(M))$,  $C(M)=C(T(M))$ and $r(M)=r(T(M))$, and we call them the  \textit{diameter}, the \textit{center},  and the \textit{radius}  of $M$, respectively  \cite[2.4]{Claus}.   Module  $M$ is said to be  a \textit{sink} (or \textit{source})  module,        provided the diameter walks of $T(M)$ start and end in sinks (or sources), and  $M$ is said to be a \textit{flow} module when $d(M)$ is odd. We can actually find    some $i\in \mathbb{Z}$ such that $M_0=\sigma^{-i}M$ is a sink module with the smallest possible radius  \cite[Theorem 1]{Claus}. We define the \textit{index} $\iota(M):=t$ when $M=\sigma^{t} M_0$ for some $t\in \mathbb{Z}$. 
 
 Let $M\in \modd(T(n),\Omega)$ be a sink (or source) module with center $c$ and radius $r\geq 1$. Then $M$ is said to be  \textit{complete} if for any walk $(x(0),x(1),\cdots,x(r))$
 in $ (T(n),\Omega)$ with $x(r)=c$ such that the vertex  $x(i)\in \supp(M)$ for all $i\in\{1,\cdots,r\}$, we have $\dim_k M_{x(0)}=\dim_k M_{x(1)}$. When $M$ is a flow module with radius $r\geq 1$, we say that $M$ is \textit{complete} if for any walk   $(x(0),x(1),\cdots,x(r))$ in $ (T(n),\Omega)_0$  with $x(r)\in C(M)$,  $x(r-1)\nin C(M)$ and $x(i)\in\supp (M)$ for all $i\in \{1,\cdots,r\}$, we have $\dim_k M_{x(0)}=\dim_k M_{x(1)}$. Otherwise, we call an indecomposable module $M$ \textit{incomplete} \cite[2.5]{Claus}.

\subsection{Regular components}

We give a short introduction about the Auslander--Reiten quivers for the finite-dimensional $k$-algebras.  Readers can refer to \cite{Bongartz} and \cite[\Romannum{4}.4]{Assem}  for locally bounded categories.

 Let $\cA$ be a finite-dimensional $k$-algebra, and let $\modd \cA$ be the category of all finite-dimensional left modules of $\cA$.  Let $X, Y\in \modd \mathcal{A}$ be indecomposable modules. We say that a homomorphism $f: X\rightarrow Y$ is  \textit{irreducible},   provided $f$ is neither a split monomorphism nor a split epimorphism, and if $f=f_1\circ f_2$, then either $f_1$ is a split epimorphism or $f_2$ is a split monomorphism.
 Let $\rad_{\cA}$ be the radical of  $\modd \cA$. Then we  define the set
\begin{center}
$\Irr(X,Y):=\rad_{\cA}(X,Y)/\rad^2_{\cA}(X,Y)$
\end{center}
 and call it the \textit{space of irreducible morphisms} \cite[\Romannum{4}.1.6]{Assem}. In fact, a homomorphism $f: X\rightarrow Y$ is  irreducible if and only if $f\in \rad_{\cA}(X,Y)\backslash \rad^2_{\cA}(X,Y)$.

\begin{Definition}
 The \textit{Auslander--Reiten quiver} (\textit{translation quiver}) $\Gamma(\mathcal{A})$ of $\mathcal{A}$  is defined as follows:
 
 \begin{enumerate}
 \item The points $\Gamma(\mathcal{A})_0$ correspond to the isomorphism classes $[X]$ of indecomposable modules $X$ in $\modd\mathcal{A}$.
 
 \item Let $[X],[Y]$ be two points in  $\Gamma(\mathcal{A})_0$. Then there are $\dim_k\Irr(X,Y)$ arrows from $[X]$ to $[Y]$ in $\Gamma(\mathcal{A})_1$.

 \end{enumerate}

\end{Definition}
Let $X, Y\in\modd \cA$ be two indecomposable modules. We say that $X$ is a \textit{predecessor} of $Y$, provided there exists a  path from $[X]$ to $[Y]$ in $\Gamma(\cA)$, i.e. a
chain of irreducible maps from $X$ to $Y$. We say that $X$ is a \textit{sucessor} of $Y$, provided there exists a  path from $[Y]$ to $[X]$ in $\Gamma(\cA)$, i.e. a chain of irreducible maps from $Y$ to $X$. We use the \textit{cones} $(X\rightarrow)$ to denote the set of all sucessors of $X$ and $(\rightarrow Y)$ all predecessors of $Y$, respectively. 

Let $X$ be a    non-projective indecomposable module (or let $Y$ be a  non-injective indecomposable module). Then there exists  a uniquely  determined short exact sequence, called \textit{Auslander--Reiten sequence}, or simply \textit{AR--sequence} (or \textit{almost split sequence})

\begin{center}
$0\rightarrow Y\xrightarrow{f} \bigoplus^t_{i=1} M^{n_i}_i \xrightarrow{g} X\rightarrow 0$,
\end{center}
where $Y$ ($X$, respectively) is indecomposable, modules $M_i$ are pairwise non--isomorphic and indecomposable. Suppose that $f=\begin{bmatrix}
f_1 \\
\vdots\\
f_t
\end{bmatrix}, g=\begin{bmatrix}
g_1,\cdots,g_t
\end{bmatrix},$ where $f_i=\begin{bmatrix}
f_{i1} \\
\vdots\\
f_{in_i}
\end{bmatrix},  g_i=\begin{bmatrix}
g_{i1},\cdots,g_{in_i}
\end{bmatrix} $. Then the maps $f_{i1}, f_{i2}, \cdots, f_{i{n_i}}: Y\rightarrow M_i$  and $g_{i1}, \cdots, g_{i{n_i}}: M_i\rightarrow X$ correspond to  bases of  $\Irr(Y,M_i)$ and $\Irr(M_i, X)$, respectively. We write $Y=\tau X$ (or $X=\tau^-Y$) and we denote this in $\Gamma(\mathcal{A})$ by $[Y]\dashleftarrow [X]$. Note that we sometimes don't distinguish between $X$ (or $Y$) and its isomorphism class $[X]$ (or $[Y]$) in $\Gamma(\cA)$. Moreover, we say that an indecomposable module $M\in \modd\cA$ \textit{preprojective} (\textit{preinjective}) if there exists $m\in \mathbb{N}_0$ such that $\tau^m M$ ($\tau^{-m} M$, respectively) is projective (injective, respectively), otherwise,  $M$ is said to be \textit{regular}.

We call a connected component $\cC$ of $\Gamma(\cA)$ \textit{regular} if it consists of regular modules. It is well-known that the regular components of $\modd \mathcal{K}_n$ are  of the type $\mathbb{Z}A_{\infty}$ (cf.\cite{Claus1}), and they are of the following form (see FIGURE \ref{Fig:RCGM}).
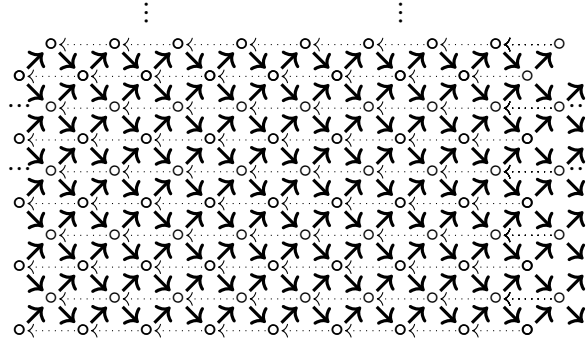
\begin{figure}[!h]

\begin{center}

\begin{tikzpicture}[very thick,scale=0.7]

                    [every node/.style={fill, circle, inner sep = 1pt}]

%%%%%%%%%%%%%%%%%%%%%%%%%%%%%% Parameter %%%%%%%%%%%%%%%%%%%%%%%%%%%%%%%%%%%

\def \n {8} % #Knoten Reihe  - 1

\def \m {4} % #Knoten Spalte - 1

\def \translation {1} % 1 Für Translation

\def \ab {0.15} % Abstand Pfeil und Knoten

\def \Pab {0.6} % Halber Abstand Horizontal

\def \lcone {1} % 1 für linken Kegel

\def \ldist {3} % Anzahl der quasi-einfachen die eingeschlossen werden - 1

\def \lcolor {red} % white für keine Farbe

\def \rcone {1} % 1 für rechten Kegel

\def \rdist {3} % Anzahl der quasi-einfachen die eingeschlossen werden - 1

\def \rcolor {red} %  white für keine Farbe

\def \llcone {0} % 1 für einen zweiten linken Kegel rechts von lcone

\def \lldist {4} % Anzahl der quasi-einfachen die eingeschlossen werden - 1

\def \rrcone {0} %1 für einen zweiten rechten Kegel links von rcone

\def \rrdist {4} % Anzahl der quasi-einfachen die eingeschlossen werden - 1

%%%%%%%%%%%%%%%%%%%%%%%%%%%% Quellcode %%%%%%%%%%%%%%%%%%%%%%%%%%%%%%%%%%%%%%

\foreach \a in {0,...,\n}{

\foreach \b in {0,...,\m}{

    \ifthenelse{\a = \n \and \b < \m}{

   %\node[color=black] ({\a,\b}) at ({\a*2*\Pab+\Pab},{\b*2*\Pab+\Pab}) {$\circ$};

   \node[color=black] ({\a,\b,5})at ({\a*2*\Pab},{\b*2*\Pab}) {$\circ$};

     }

     {

      \ifthenelse{\b = \m \and \a < \n}{

      \node[color=black] ({\a,\b}) at ({\a*2*\Pab+\Pab},{\b*2*\Pab+\Pab}) {$\circ$};

      \node[color=black] ({\a,\b,5})at ({\a*2*\Pab},{\b*2*\Pab}) {$\circ$};

      }

      {

      \ifthenelse{\b = \m \and \a = \n}

     {\node[color=black] ({\a,\b,5})at ({\a*2*\Pab},{\b*2*\Pab}) {$\circ$};}

    {\node[color=black] ({\a,\b}) at ({\a*2*\Pab+\Pab},{\b*2*\Pab+\Pab}) {$\circ$};

    \node[color=black] ({\a,\b,5})at ({\a*2*\Pab},{\b*2*\Pab}) {$\circ$};

      }

      }

      }

    }

    }

\foreach \s in {0,...,\n}{

\foreach \t in {0,...,\m}{  

 \ifthenelse{\t = \m \and \s < \n}{

    \draw[->] (\s*2*\Pab+\ab,\t*2*\Pab+\ab) to (\s*2*\Pab+\Pab-\ab,\t*2*\Pab+\Pab-\ab); 

    \draw[->] (\s*2*\Pab+\Pab+\ab,\t*2*\Pab+\Pab-\ab) to (\s*2*\Pab+2*\Pab-\ab,\t*2*\Pab+\ab); 

 }{

   \ifthenelse{\s = \n \and \t < \m}{

    %\draw[->] (\s*2*\Pab+\Pab+\ab,\t*2*\Pab+\Pab+\ab) to (\s*2*\Pab+2*\Pab-\ab,\t*2*\Pab+2*\Pab-\ab);

   %\draw[->] (\s*2*\Pab+\ab,\t*2*\Pab+2*\Pab-\ab) to (\s*2*\Pab+\Pab-\ab,\t*2*\Pab+\Pab+\ab);  

   %\draw[->] (\s*2*\Pab+\ab,\t*2*\Pab+\ab) to (\s*2*\Pab+\Pab-\ab,\t*2*\Pab+\Pab-\ab); 

   }

  {

  \ifthenelse{\s = \n \and \t = \m}{

    }{

   \draw[->] (\s*2*\Pab+\ab,\t*2*\Pab+\ab) to (\s*2*\Pab+\Pab-\ab,\t*2*\Pab+\Pab-\ab); 

   \draw[->] (\s*2*\Pab+\Pab+\ab,\t*2*\Pab+\Pab+\ab) to (\s*2*\Pab+2*\Pab-\ab,\t*2*\Pab+2*\Pab-\ab);

   \draw[->] (\s*2*\Pab+\ab,\t*2*\Pab+2*\Pab-\ab) to (\s*2*\Pab+\Pab-\ab,\t*2*\Pab+\Pab+\ab); 

   \draw[->] (\s*2*\Pab+\Pab+\ab,\t*2*\Pab+\Pab-\ab) to (\s*2*\Pab+2*\Pab-\ab,\t*2*\Pab+\ab);    

   }

   }

   }

    }

    }

\draw[->] (\n*2*\Pab+\ab,\m*\Pab+2*\Pab+\Pab+\Pab+\ab) to (\n*2*\Pab+\Pab-\ab,\m*\Pab+2*\Pab+\Pab+\Pab+\Pab-\ab);

\ifthenelse{\isodd{\m}}

%% IF

 { 

  \node[color=black] (Dots1) at (0,\m*\Pab+2*\Pab+\Pab) {$\cdots$};

  \node[color=black] (Dots2) at (1+\n*2*\Pab,\m*\Pab+2*\Pab+\Pab) {$\cdots$};

   \ifthenelse{\isodd{\n}}{

  \node[color=black] (Dots3) at (-0.2*\n*2*\Pab,2*\m*\Pab+3*\Pab) {$\vdots$};}

  {\node[color=black] (Dots4) at (0.75*\n*2*\Pab,2*\m*\Pab+2*\Pab) {$\vdots$};} 

  }

%% Else

  {

  \node[color=black] (Dots1) at (0,\m*\Pab+\Pab) {$\cdots$};

  \node[color=black] (Dots2) at (1+\n*2*\Pab,\m*\Pab+\Pab) {$\cdots$};

  \ifthenelse{\isodd{\n}}{

  \node[color=black] (Dots3) at (0*\n*2*\Pab,2*\m*\Pab+3*\Pab) {$\vdots$};}

  {\node[color=black] (Dots4) at (0.25*\n*2*\Pab,2*\m*\Pab+2*\Pab) {$\vdots$};}

  }

\ifthenelse{\translation = 1}{

   \foreach \s in {0,...,\n}{

   \foreach \t in {0,...,\m}{ 

   \ifthenelse{\s = 0}{}{

      \ifthenelse{\s = \n}{\draw[->,dotted,thin] (\s*2*\Pab-\ab,\t*2*\Pab) to (\s*2*\Pab-2*\Pab+\ab,\t*2*\Pab); }{

   \draw[->,dotted,thin] (\s*2*\Pab-\ab,\t*2*\Pab) to (\s*2*\Pab-2*\Pab+\ab,\t*2*\Pab); 

   \draw[->,dotted,thin] (\s*2*\Pab-\ab+\Pab,\t*2*\Pab+\Pab) to (\s*2*\Pab-2*\Pab+\Pab+\ab,\t*2*\Pab+\Pab); 

\draw[->,dotted,thin] (\n*2*\Pab-\ab+\Pab,\t*2*\Pab+\Pab) to (\n*2*\Pab-2*\Pab+\Pab+\ab,\t*2*\Pab+\Pab); 

   }

   }}

}}

{}  %ELSE

                                                                                                                                                                                             \end{tikzpicture}

\end{center}
\caption{Regular component $\mathbb{Z}A_{\infty}.$}
\label{Fig:RCGM}
\end{figure}
A module $M$  is called \textit{quasi-simple} if it is located in the bottom row of  a $\mathbb{Z}A_{\infty}$-component.  Let  $M$  be a module in a regular $\mathbb{Z}A_{\infty}$-component  $\cC$. Then there exists a  quasi-simple module $X\in \cC$ (or $Y\in \cC$)  such that  we can find a   chain of irreducible monomorphisms $X=X_1\rightarrow \cdots \rightarrow X_{s-1}\rightarrow X_s=M$ (or irreducible epimorphisms $M=Y_s\rightarrow Y_{s-1}\rightarrow \cdots \rightarrow Y_1=Y$). We call $s$ the \textit{quasi-length} of $M$ and the module $X$ (or $Y$) \textit{quasi-socle} (or \textit{quasi-top}) of $M$.  It can be shown that  $M$ is uniquely determined by its quasi-length and quasi-socle (or  quasi-top), whence we can define $ql(M):= s=$ the quasi-length of $M$.

Given a regular $\mathbb{Z}A_{\infty}$-component $\cC$,  there are uniquely determined quasi-simple modules $M_{\cC}$ and $W_{\cC}$ in $\cC$ such that the cone $(M_{\cC}\rightarrow)$ of all successors of $M_{\cC}$ satisfies $(M_{\cC}\rightarrow)=\EKP_n     \cap\cC$ and the cone $(\rightarrow W_{\cC})$ of all predecessors of $W_{\cC}$ satisfies $(\rightarrow W_{\cC})=\EIP_n\cap\cC$ \cite[Theorem 3.3]{Julia}. Hence we can measure the distance between $M_{\cC}$ and $W_{\cC}$.
 
\begin{Definition}
Let $\cC$ be a regular AR component of $\modd \cK_n$. Define   an integer $\cW(\cC)$ satisfying $\tau^{\cW(\cC)+1}M_{\cC}=W_{\cC}$.  We call $\cW(\cC)$  the \textit{width} of $\cC$.

\end{Definition}

We now define

\begin{center}
$\Inj_{n,\Omega} := \{M\in\modd(T(n),\Omega)\mid \forall \delta \in (T(n),\Omega)_1: M(\delta)$ is injective $\}$,

$\Sur_{n,\Omega} :=\{M\in \modd(T(n),\Omega)\mid \forall \delta \in (T(n),\Omega)_1: M(\delta)$ is surjective $\}.$ 

\end{center}
Let $M\in\modd(T(n),\Omega)$ be an indecomposable module. By \cite[Theorem 8.1]{Daniel}, we know that $M\in \Inj_{n,\Omega}$ if and only if $\pi_\lambda(M)\in\EKP_n$.

A component  of category $\modd(T(n),\Omega)$ is said to be \textit{regular} if it contains a regular module. Let $\cD$  be a regular component of $\modd(T(n),\Omega)$. According to \cite[Corollary 7.2]{Daniel},  the map $\cD \rightarrow \pi_\lambda(\cD)$  is an isomorphism of  translation quivers. Consequently, the regular component $\cD$ is also of the type $\mathbb{Z}A_{\infty}$. Hence it is similar with that of the category $\modd \cK_n$, and there  exist unique quasi-simple modules $M_{\cD}, W_{\cD}$ in $\cD$ such that $(M_{\cD}\rightarrow)=\Inj_{n,\Omega}\cap \cD$ and $(\rightarrow W_{\cD})=\Sur_{n,\Omega}\cap  \cD$. Then we have the following

\begin{Definition}
Let $\cD$ be a regular AR component of $\modd(T(n),\Omega)$. We define an integer $\cW(\cD)$  satisfying $\tau^{\cW(\cD)+1}M_{\cD}=W_{\cD}$. We call $\cW(\cD)$  the \textit{width} of $\cD$.

\end{Definition}

\section{Graded Modules in regular components}

We use $\cR$ to denote the set of regular AR components of $\modd(T(n),\Omega)$, and we now focus on a component $\cD\in \cR$ in this section. Let $\sigma \cD=\sigma^-\cD$ denote the corresponding component under the action of reflection functors $\sigma$ and $\sigma^-$. Let $M\in\modd(T(n),\Omega)$ be a regular indecomposable module. We  define  $\mathcal{N}_M(b): =\mathcal{N}(b)\cap T(M)_0, \forall b\in T(M)_0$.

\begin{Lemma}\label{sink length}

Suppose that $M\in\Inj _{n,\Omega}$  is a regular indecomposable module. Then
 \begin{enumerate}
\item $M$ is a sink module and $d(M)\geq 4$.
 
 \item  Any leaf of $T(\sigma M)$ is a sink. In particular, $\sigma M$ is a sink module. 
 \end{enumerate}

\end{Lemma}
\begin{proof}
According to the definition, every map is injective in $M$, that is, every leaf of $T(M)$ is a sink,  whence $M$ is a sink module. Note that $\sigma M$ is also indecomposable. Suppose that there exists a source  $b_0$ being  a leaf in  $T(\sigma M)$. Let $\alpha: b_0\rightarrow b_1$ be an arrow of  $T(\sigma M)_1$ and the $k$-linear map  $\varphi=(\sigma M)(\alpha): (\sigma M)_{b_0}\rightarrow (\sigma M)_{b_1}$ be non-trivial.  Since $\sigma^-(\sigma M)\cong M$, there exists a map  $\varphi'=M(\alpha'): M_{b_1}=(\sigma M)_{b_1}\rightarrow M_{b_0}=   \coker {\varphi}$ that is not injective in $M$, this is a contradiction.  Hence all leaves of  $T(\sigma M)_0$ are sinks  and $\sigma M $ is a sink module. Moreover, $M$ is complete and $d(\sigma M)=d(M)-2$, $C(\sigma M)=C(M)$ \cite[3.1]{Claus}.

 Since  $M$ is a  sink  and regular module, it follows that $d(M)\geq 4$: if $d(M)=2$, then we have  $d(\sigma M)=0$, that is, $\mid T(\sigma M)_0\mid =1$,  so that $\sigma M $ is not regular, a contradiction.

\end{proof}

\begin{corollary}\label{source length}

  Suppose that $N\in\Sur_{n,\Omega}$ is a regular indecomposable module. Then
  
  \begin{enumerate}
 \item $N$ is a source module and $d(N)\geq 4$.
 
 \item Any leaf of $T(\sigma^- N)$ is a source. In particular, $\sigma^- N$ is a source module. 
 \end{enumerate}

\end{corollary}

\begin{proof}
Since $N\in\Sur_{n,\Omega}$, we obtain  $D_{T(n)}(N)\in\Inj_{n,\sigma \Omega}$. By Lemma \ref{sink length}, we know that $N$ is a source module and $d(N)\geq 4$. Similarly,  we get  $(b)$.

\end{proof}

We give an example to see how the modules change in their $\sigma$-orbits. 
\begin{example}
 Let $N= k\xleftarrow{\lambda_1} k\xrightarrow{\lambda_2} k\in\modd(T(3),\Omega)$, where $\lambda_1,\lambda_2\in k\setminus\{0\}$. Then $N$ is indecomposable \cite[Proposition 1]{Claus5}. Actually,  $N$ is regular  since $\pi_\lambda(N)$ is regular. Let $\cD$ be a regular component. Without loss of generality, we assume that $N\in \cD$.  Let $M=\sigma^{-2}N$. It is easy to check that  $M\in \Inj_{3,\Omega}\cap \cD$. Note that $N\nin \Inj_{3,\Omega}\cap \cD$ and $N$ is quasi-simple \cite[Proposition 3.4]{Bo}. Then $M$ is quasi-simple and $M=M_{\cD}$. One can check that $\sigma^- M$ is a sink module.

\end{example}

\begin{Lemma}\label{M}
 Suppose that $M\in\Inj_{n,\Omega}\cap \cD$  is a regular indecomposable module. Then   $\sigma^- M\in\Inj_{n,\sigma\Omega}\cap  \sigma \cD$.
\end{Lemma}

\begin{proof}

 Let $M'=\sigma^{-}M$.  We  want to prove that $M'\in$ Inj$_{n,\sigma\Omega}\cap \sigma \cD$. Since $M$ is a sink module,  all neighbours of any source $x \in T(M)_0$ still belong to $T(M)_0$. We have injective map $f_i : M_x \rightarrow M_{y_i}$ for every $ i\in\{ 1,\cdots, n\}$ and $y_i\in \mathcal{N}_{M}(x)$. Put $ M’_x:= (\bigoplus_{y_i\in \mathcal{N}_{M}(x)} M_{y_i})/ \im f$, where $f(m)= \begin{bmatrix}
f_1(m)\\
\vdots\\
f_n(m)
\end{bmatrix},  m\in M_x$. For each $i$, we have to show that the map
\begin{center}
$g_i : M'_{y_i}=M_{y_i} \rightarrow M'_x  ; m_i \mapsto (0, \cdots,0,\overline{m}_i,0, \cdots0)$
\end{center}
is injective, where $g_i$ maps $m_i$ to the $i$-th position in $M'_x$ with the image $\overline{m}_i$. Now let $m_i \in \ker   g_i$. Then $ (0,\cdots,0, \overline{m}_i, 0, \cdots, 0) \in \im f$. Hence there exists $m \in M_x$ such that $m_i=f_i(m)$, while $f_j(m)=0$ for all $j\ne i$. Since each $f_j$ is injective
(actually,  $\bigcap_{j\ne i}\ker f_j=(0)$ suffices), it follows that $m=0$, whence $m_i=0$. Consequently, $g_i$ is injective.

\end{proof}

\begin{Lemma}\label{N}
Suppose that $N\in\Sur_{n,\Omega}\cap \cD$ is a regular indecomposable module. Then $\sigma N\in\Sur_{n,\sigma\Omega}\cap \sigma \cD$.

\end{Lemma}
\begin{proof}
Since $D_{T(n)}(N)\in\Inj_{n,\sigma\Omega}$,  it follows from Lemma \ref{M} that  $\sigma^- \circ D_{T(n)}(N)\cong D_{T(n)}\circ \sigma(N)\in$ Inj$_{n,\sigma\Omega}$, that is, $\sigma N\in$ Surj$_{n,\sigma\Omega}\cap \sigma\cD$. 

\end{proof}

 Let $M\in \cD$ be a regular indecomposable module. Recall that   $ql(M)$  is the quasi-length of $M$ in $\cD$.  Then we have the following.

 \begin{Lemma}\label{MD}

 $\iota(M_{\sigma \cD})\in \{\iota(M_{\cD})-1, \iota(M_{\cD})+1\}$.

\end{Lemma}
\begin{proof}
 Let $M=M_{\cD}$, and let $N\in\Inj_{n,\Omega}\cap \cD$ with $ql(N)=1$. Since $M$ is regular with $ql(M)=1$,  we get  $N=\tau^{-t} M=\sigma^{-2t} M$ and $\iota(N)=\iota(M)-2t$ for some $t\in \mathbb{N}_{0}$ in the regular component $\cD$. Hence $M$ is the one with the biggest index $\iota(M)$ in its $\sigma$-orbit   located in $\Inj_{n,\Omega}\cap  \cD$.  On the other hand, we have $\sigma^{-}M \in\Inj_{n,\sigma\Omega}\cap \sigma \cD$  by Lemma \ref{M}. Moreover, we have $ ql(\sigma^{-} M)=ql(M_{\sigma \cD})=ql(M)$, and $\iota(\sigma^-M)\leq \iota (M_{\sigma \cD})$. Then  $\iota(M_{\sigma \cD})\geq \iota (M_{\cD})-1$ and $\iota(M_{\sigma \cD})\neq \iota(M_{\cD})$. Otherwise,  suppose that $\iota(M_{\sigma \cD}) \leq \iota(M)-3$. Then $\iota(\sigma^-M)=\iota(M)-1>\iota(M_{\sigma \cD})$, a contradiction. Similarly,  suppose that $\iota(M_{\sigma \cD})\geq \iota(M)+3$. We have $\iota(\sigma^- M_{\sigma \cD})= \iota(M_{\sigma \cD})-1>\iota(M)$ and $\sigma^- M_{\sigma \cD}\in$ Inj$_{n,\Omega}\cap \cD$ by Lemma \ref{M},  a contradiction.  Hence  $\iota(M_{\sigma \cD})\in \{ \iota(M)-1, \iota(M)+1\}$.
\end{proof}

\begin{Lemma}\label{WD}

$\iota(W_{\sigma \cD})\in\{ \iota(W_{\cD})-1, \iota(W_{\cD})+1\}$.

\end{Lemma}

\begin{proof}
Let $U\in\Sur_{n,\Omega}\cap \cD$ with $ql(U)=1$. Similarly,  we can see that  $U=\tau^s W_{\cD}=\sigma^{2s}W_{\cD}$ and $\iota(U)=\iota(W_{\cD})+2s$ for some $s\in \mathbb{N}_{0}$. Hence $\iota(U)\geq \iota(W_{\cD})$. By Lemma \ref{N},  we have $\sigma W_{\sigma \cD}\in\Sur_{n,\Omega}\cap \cD$.   Suppose that $\iota(W_{\sigma \cD})\leq  \iota(W_{\cD})-3$. Then $\iota(\sigma W_{ \sigma \cD})=\iota(W_{\sigma \cD})+1<\iota(W_{\cD})$, a contradiction.     Suppose that $\iota ( W_{\sigma \cD})\geq \iota(W_{\cD})+3$. Then $\iota(\sigma W_{\cD})=\iota(W_{\cD})+1<\iota(W_{\sigma \cD})$ and $\sigma W_{\cD}\in$ Surj$_{n,\sigma\Omega}\cap \sigma \cD$  by Lemma \ref{N}, a contradiction. Hence $\iota(W_{\sigma \cD})< \iota(W_{\cD})+3$. Finally, we have $\iota(W_{\sigma \cD})=\iota(W_{\cD})+2s'+1$ for some $s'\in\mathbb{Z}$.
\end{proof}

Let $M\in \cD$ be a regular indecomposable module,  and let $b(M)$ be  the number of flow modules in the $\sigma$-orbit of $M$. Then  $b(M)=b(N)$ for any $N\in\cD$   \cite[Section 7]{Claus}.  Hence we can  define $b(\cD):=b(M)$.

\begin{Theorem}\label{D}
Let $\cD$ be a regular component of $\modd(T(n),\Omega)$. Then 
\begin{center}
$ \mid \cW(\cD) -\cW(\sigma \cD)\mid \leq 1 $ and $\frac{b(\cD)+1}{2}\leqslant \cW(\cD)$.
\end{center}
 
\end{Theorem}

\begin{proof}
Since  $\tau=\sigma^2$,  we have $(\sigma^2)^{(\cW(\cD)+1)} M_{\cD}=W_{\cD}$.   Then $2(\cW(\cD)+1)+\iota(M_{\cD})=\iota(W_{\cD})$, that is, $\cW(\cD)=\frac{\iota(W_{\cD})-\iota(M_{\cD})}{2}-1 $. According to Lemma \ref{MD} and Lemma \ref{WD}, we have inequalities
\begin{equation}
\begin{split}
 \frac{\iota(W_{\cD})-1-(\iota(M_{\cD})+1)}{2}-1  \leq \cW(\sigma D)=\frac{\iota(W_{\sigma \cD})-\iota(M_{\sigma \cD})}{2}-1  \leq \frac{\iota(W_{\cD})+1-(\iota(M_{\cD})-1)}{2}-1.
\end{split}
\end{equation}
 
   Thus
\begin{equation}\label{1}
 \cW(\cD) -1 \leq \cW(\sigma \cD)\leq \cW(\cD)+1.
\end{equation}

According to Lemma \ref{sink length} and Corollary \ref{source length},  there are only sink modules contained in $\Inj_{n,\Omega}$ and   source modules  contained  in  $\Sur_{n,\Omega}$. Let $N$  be a regular indecomposable module  in the  $\sigma$-orbit of the module $M_{\cD}$.  Then   $N=\sigma^q M_{\cD}$ for some $q\in \mathbb{N}_{0}$. Suppose that $N$ is a flow module. Since we already know that $\sigma M_{\cD}$ is a sink module and $\sigma^-W_{\cD}$ is a source module, we have
\begin{center}
$\iota(M_{\cD})+1=\iota(\sigma M_{\cD})<\iota(N)=q+\iota(M_{\cD})<\iota(\sigma^-W_{\cD})=\iota(W_{\cD})-1$. 
\end{center}
Hence $1<q < \iota(W_{\cD})-1-\iota(M_{\cD})= 2 \cW(\cD)+1$, that is, $q\in \{2,3, \cdots,2\cW(\cD)\}$.   Then  the number of flow modules $b(N)\leq 2\cW(\cD)-1$. However,   $b(\cD)=b(N)$.
Hence  $\frac{b(\cD)+1}{2}\leqslant \cW(\cD)$.

\end{proof}

\begin{corollary}\label{cor}
Let $\cD$ be a regular component of $\modd(T(n),\Omega)$. Suppose that $\cW(\cD)=\cW(\sigma \cD)$. Then $\cW(\cD)\geq \frac{b(\cD)+2}{2}$.

\end{corollary}
\begin{proof}
We  define two sets
\begin{center}
$\cM:=\cD \setminus (\Inj_{n,\Omega}\cup \Sur_{n,\Omega})$, and  $\cM(\sigma):=\sigma \cD \setminus (\Inj_{n,\sigma\Omega}\cup \Sur_{n,\sigma\Omega})$.
\end{center}

If  $\cW(\cD)\leq \frac{b(\cD)+1}{2}$, then we have $\cW(\cD)= \frac{b(\cD)+1}{2}$ by Theorem \ref{D}. Hence we get $\cW(\cD)=\cW(\sigma\cD)=\frac{b(\cD)+1}{2}$. That is to say, there are $b(\cD)+1$ quasi-simple modules contained in the set $\cM \cup \cM(\sigma)$. However, there exist at least one sink  module and one source  module that are quasi-simple in the set $\cM \cup \cM(\sigma)$ by Lemma \ref{sink length} and Corollary \ref{source length}. Hence there are at least $b(\cD)+2$ quasi-simple modules contained in the set $\cM \cup \cM(\sigma)$, and this is a contradiction.

\end{proof}

We  now give an example with $b(\cD)=0$.
\begin{example}

 Let $M\in\modd(T(3),\Omega)$ be the module
\begin{center}

\begin{tikzpicture}
\node (00) at (0,0) {$0$};
\node (10) at (1,0) {$k$};
\node (20) at (2,0) {$k$};
\node (30) at (3,0) {$k$};
\node (40) at (4,0) {$k$};
\node (50) at (5,0) {$k$};
\node (60) at (6,0) {$0$ ,};
\node (1-1) at (1,-1) {$0$};
\node (2-1) at (2,-1) {$k$};
\node (3-1) at (3,-1) {$k$};
\node (4-1) at (4,-1) {$k$};
\node (5-1) at (5,-1) {$0$};
\node (2-2) at (2,-2) {$k$};
\node (4-2) at (4,-2) {$k$};
\path [->] (00) edge (10)
           (20) edge node[above]{$\lambda_1$}(10) 
           (1-1) edge (10)
          (20) edge node[above]{$\lambda_2$} (30)
            (40) edge node[above]{$\lambda_3$} (30)               
             (3-1) edge node[midway, right]{$\lambda_6$} (30)              
               (20) edge node[midway, right]{$\lambda_5$} (2-1)            
                (40) edge node[above]{$\lambda_4$} (50)
                (40) edge node[midway, right]{$\lambda_7$} (4-1)
                (60) edge (50)
                (5-1) edge (50)
              (3-1) edge node[midway, above]{$\lambda_8$} (2-2)  
                (3-1) edge   node[midway, left]{$\lambda_9$}   (4-2)
                          
                           ;

\end{tikzpicture}

\end{center}
where $\lambda_i\in k\setminus \{0\}, i\in \{1,2,\cdots,9\}$.
Then $M$ is indecomposable and $M\in\Inj_{3,\Omega}$ \cite[Proposition 1]{Claus5}. Actually,  $M$ is  quasi-simple   \cite[Proposition 3.4]{Bo}.  We can see that $M$ is regular since $\pi_\lambda(M)$ is regular. Furthermore, we have $b(M)=0$ and $M=M_{\cD}$ for some regular component $\cD$ of $\modd(T(3),\Omega)$. It is easy to see that  $\sigma^3 M, \sigma^4 M$ are  source modules and  maps in them all are  surjective. Since   $\sigma^2 M\nin\Sur_{3,\Omega}$, we obtain $\sigma^4 M=\tau^2 M=W_{\cD}$ and $\cW(\cD)=1$. Moreover, $\sigma M\nin\Inj_{3,\sigma\Omega}\cup\Sur_{3,\sigma\Omega}$, whence $\cW(\sigma \cD)=1$ by Lemma \ref{M}.

\end{example}

\begin{proposition}
 $\{\cW(\cD)|\cD\in \cR\}=\mathbb{N}$.
\end{proposition}
\begin{proof}
According to Theorem \ref{D}, we get $\cW(\cD)\geq \frac{b(\cD)+1}{2} >0$ for any $\cD\in \cR$. Then $\{\cW(\cD)| \cD\in \cR\}\subseteq \mathbb{N}$. However, we already have $\mathbb{N}\subseteq \{\cW(\cD)| \cD\in \cR\}$ by \cite[Theorem 10.3$(b)$]{Daniel}. Hence $\{\cW(\cD)|\cD\in \cR\}=\mathbb{N}$.

\end{proof}

\begin{Remark}
If there existed a regular component $\cD$ of $\modd(T(n),\Omega)$ such that $\cW(\cD)=0$, then $\tau M_{\cD}=W_{\cD}=\sigma^2M_{\cD}$. Hence the module $\sigma M_{\cD}$ would be a sink module and the module $\sigma^-W_{\cD}\cong \sigma M_{\cD}$ would be a source module by the Lemma \ref{sink length}$(b)$ and Corollary \ref{source length}$(b)$. Apparently, this is a contradiction. 

\end{Remark}

\clearpage

%%%%%%%%%%%%%%%%%%%%%%% REFERENCES %%%%%%%%%%%%%%%%%%%


\begin{bibdiv}
\begin{biblist}
\addcontentsline{toc}{chapter}{\textbf{Bibliography}}
\bib{Julia}{article}{
title = {Categories of modules for elementary abelian p-groups and generalized Beilinson algebras},
author = {J. Worch},
journal = {J. London Math. Soc.},
date = {2013},
volume = {88},
pages = {649-688},
}


\bib{Luis}{article}{
   author={L. \'{A}lvarez-C\'{o}nsul},
   author={A. King},
   title={Moduli of sheaves from moduli of Kronecker modules},
   conference={
      title={Moduli spaces and vector bundles},
   },
   book={
      series={London Math. Soc. Lecture Note Ser.},
      volume={359},
      publisher={Cambridge Univ. Press, Cambridge},
   },
   isbn={978-0-521-73471-4},
   date={2009},
   pages={212--228},
  }
\bib{Claus4}{article}{

 AUTHOR = {C. Ringel},
     TITLE = {Quiver {G}rassmannians for wild acyclic quivers},
   JOURNAL = {Proc. Amer. Math. Soc.},
  FJOURNAL = {Proceedings of the American Mathematical Society},
    VOLUME = {146},
      YEAR = {2018},
    NUMBER = {5},
     PAGES = {1873--1877},
      ISSN = {0002-9939,1088-6826},
   MRCLASS = {16G20 (14D20 16G60)},
  MRNUMBER = {3767342},
MRREVIEWER = {Thorsten\ Weist},
        }

\bib{Assem2}{book}{
title={Elements of the representation Theory of Associative Algebras, \Romannum{3}},
subtitle={Representation-infinite Tilted Algebras},
series={London Math. Soc. Student Tex.},
author={D. Simson},
author={A. Skowro\'nski},
publisher={Cambridge Univ. Press},
date={2007},
address={Cambridge},
}

\bib{Kerner}{article}{
title={Representations of Wild Quivers Representation theory of algebras and related topics},
author={Kerner, O.},
journal={CMS
Conf. Proc.},
volume={19},
date={1996},
pages={65-107},

}


\bib{Claus}{article}{
title = {The shift orbits of the graded Kronecker modules},
author = {C. Ringel},

journal = {Math. Z.},
volume = {290},
date = {2018},
pages = {1199-1222},
number = {3},
}




\bib{Claus3}{webpage}{

title={Covering Theory}
author = {C. Ringel},
url={https://www.math.uni-bielefeld.de/~ringel/lectures/izmir/izmir-6.pdf}
}


\bib{Bongartz}{article}{
title={Covering spaces in representation theory},
author={K. Bongartz and P. Gabriel},
journal={Invent. Math. },
volume={65},
date={1981/82},

pages={331-378},
}
\bib{Claus1}{article}{
title = {Finite-dimensional hereditary algebras of wild representation type},
author = {C. Ringel},

journal = {Math. Z.},
volume = {161},
date = {1978},
pages = {235-255},

}
\bib{Assem}{book}{
title={Elements of the representation Theory of Associative Algebras, \Romannum{1}},
subtitle={Techniques of Representation Theory},
series={London Math. Soc. Student Tex.},
author={I. Assem},
author={D. Simson},
author={A. Skowro\'nski},
publisher={Cambridge Univ.  Press},
date={2007},
address={Cambridge},
}



\bib{Daniel}{thesis}{
title = {Representations of Regular Trees and Invariants of AR-Components for Generalized Kronecker Quivers},
author = {D.  Bissinger},
date = {2018},
school={Mathematisch-Naturwissenschaftliche
Fakultät, Christian-Albrechts-Universität zu Kiel},
type={PhD thesis},
url = {https://macau.uni-kiel.de/servlets/MCRFileNodeServlet/dissertation_derivate_00007342/DissertationDanielB.pdf}
}

























\bib{Claus5}{webpage}{

title={Simple representations, thin representations}
author = {C. Ringel},
url={https://www.math.uni-bielefeld.de/~sek/kau/leit2v2.pdf}
}





\bib{Bo}{article}{
  title = {Dimension vectors in regular components over wild Kronecker quivers},
  author = {B. Chen},
  journal = {Bull. Sci. Math.},
  volume = {137},
  number = {6},
  pages = {730-745},
  year = {2013},
  publisher = {Elsevier},
 } 

\bib{Zhang}{article}{
  title = {The modules in any component of the $AR$--quiver of a wild hereditary Artin algebra are uniquely determined by their composition factors},
  author = {B. Zhang},
  journal = {Acta. Math. Sin.},
  volume = {6},
pages = {97-99},
  year = {1990},
 
 } 




\bib{Jie}{thesis}{
title = {Modules of  generalized Kronecker quivers},
author = {J. Liu},
date = {2021},
school={Mathematisch-Naturwissenschaftliche
Fakultät, Christian-Albrechts-Universität zu Kiel},
type={PhD thesis},
url = {https://macau.uni-kiel.de/servlets/MCRFileNodeServlet/macau_derivate_00002730/Jie_Liu.pdf}
}

















\end{biblist}
\end{bibdiv}
\end{document}